\documentclass[smallextended,natbib]{svjour3}       
\journalname{Beitr\"age zur Algebra und Geometrie}

\usepackage{linenum}
\def\linenumber{\tiny\sffamily
A\number\linecount}
\let\resetlinecount=\relax

\usepackage{amsmath,amsfonts}
\usepackage{tikz}
\usetikzlibrary{calc,decorations.markings}
\usepackage{url}
\usepackage{hyperref}
\usepackage{microtype}
\newcommand{\eps}{\varepsilon}

\ifnum\mag>1000
\fi

\begin{document}

\addtolength{\textheight}{-3mm}

\title{
Optimal Triangulation of Saddle
 Surfaces
\thanks{%
Dror Atariah and G\"unter Rote were
supported by Deutsche Forschungsgemeinschaft (DFG) within the DFG Collaborative Research Center TRR 109 \emph{Discretization in Geometry and
Dynamics}.
Mathijs Wintraecken
was supported by the Future and Emerging Technologies
(FET) program within the Seventh Framework
Program for Research of the European Commission,
under FET-Open grant number 255827, as part of the
project \emph{Computational Geometric Learning}.
Partial support has been provided by the Advanced Grant of the European Research Council GUDHI (\emph{Geometric Understanding in Higher Dimensions}).
}}

\author
{Dror Atariah\and
G\"unter Rote\and
Mathijs \rlap{Wintraecken}}
\authorrunning
{Dror Atariah,
G\"unter Rote, and
Mathijs Wintraecken}

\institute
{Dror Atariah\at
Freie Universit\"at Berlin, Institut f\"ur Informatik,
\email{drorata@gmail.com}           
\and
G\"unter Rote
\at
Freie Universit\"at Berlin, Institut f\"ur Informatik,
\email{rote@inf.fu-berlin.de}
\and
Mathijs Wintraecken
\at
INRIA Sophia Antipolis, Geometrica,
\email{m.h.m.j.wintraecken@gmail.com}
}

\date{\today} %
\date{July 27, 2017}
\maketitle
\begin{abstract}
\noindent
  We consider the piecewise linear approximation of
  saddle functions of the form $f(x,y)=ax^2-by^2$ under the
$L_{\infty}$ error norm. We show that interpolating approximations
are not optimal. One can get slightly 
smaller errors 
 by
allowing the vertices of the approximation to move away from the graph
of the function.

\end{abstract}










\section{Introduction}

We are given the 
 bivariate quadratic function
\begin{equation}
  \label{eq:f}
f(x,y)=ax^2+2bxy+cy^2+dx+ey+g  ,
\end{equation}
and we want to approximate it with a piecewise linear function $\hat
f(x,y)$ which is as simple as possible. 
More precisely, the function $\hat f$
that we are looking for is
defined by a triangulation $\mathcal{T}$ of the plane and the values
$\hat f(x,y)$ at the vertices $(x,y)
$ of the
triangulation.
 We
want to minimize the number of triangles.  The error criterion that we
consider is the $L_\infty$ distance or \emph{vertical distance}
(thinking geometrically in the three-dimensional space where the graph
of $f$ lives); it 
should be bounded by some specified parameter $\eps$:
\begin{equation*}
\max_{x,y} | f(x,y)-\hat f(x,y) |
\le \eps
\end{equation*}
For simplicity, we will let $(x,y)$ range over the whole plane.
Thus, we cannot just count the triangles. We rather minimize the
\emph{triangle density}. Let $Q_r=[-r/2,r/2]\times [-r/2,r/2]$ be the
$r\times r$ square centered at the origin.
The triangle density counts the number of triangles
$T\in \mathcal{T}$ of the triangulation that
 intersect
the squares $Q_r$ for larger and larger side length $r$,
 in comparison to the area of these squares:
\begin{equation*}
\limsup_{r\to\infty}
\frac{| \{\,T\in \mathcal{T} \mid 
T \cap Q_r \ne \emptyset\,\}|}
{r^2}
\end{equation*}

We have three cases:
\begin{enumerate}
\item 
 $f$ is a positive or negative definite quadratic function; in other words, $f$ is
convex or concave.
\item 
 $f$ is indefinite; the graph of $f$ is a saddle surface, and it has negative Gauss curvature.
\item 
 $f$ is semidefinite; its graph is a parabolic cylinder.
\end{enumerate}
The cases can be distinguished by the 
 discriminant $ac-b^2$
being positive, negative, or zero.

Case~1 
 is easy; the classical theory of piecewise linear convex
approximation applies.
 We will mention the respective results below.
Case~3, 
as well as the case of a linear function ($a=b=c=0$), is a
boundary case, and we will not treat it. 
We will concentrate on
Case~2, 
which is representative of negatively curved surfaces in 3
dimensions:

\begin{theorem}\label{main}
  If $f$ is indefinite $(ac-b^2<0)$, then there is a piecewise
  linear function $\hat f$ approximating $f$ with vertical error
  $\eps$ that has triangle density
$$
 \frac{\sqrt 3}{4}
\cdot\sqrt{b^2-ac}
\cdot \frac1\eps
 \approx 0.43301\cdot \sqrt{b^2-ac}\bigm/\eps.
$$
The triangulation consists of a grid of congruent triangles like in
Figure~\ref{grid}. This grid can be freely translated in the plane,
and in addition, there is a one-parameter family of solutions of
different shapes with the
same properties.
\end{theorem}

If we require an \emph{interpolating}
approximation, the error of the approximation is slightly larger,
as stated in the following theorem.
Here, the vertices of $\hat f$
are constrained to lie on the given surface, i.e., $\hat f(x,y)=f(x,y)$
for
all vertices $(x,y)
$ of the
triangulation.
%
\begin{theorem}
\label{interpolating}
  If $f$ is indefinite $(ac-b^2<0)$, then there is a piecewise
  linear interpolating function $\hat f$ approximating $f$ with vertical error
  $\eps$ that has triangle density
$$
 \frac{1}{\sqrt{5}}\cdot\sqrt{b^2-ac}\cdot \frac1\eps
 \approx 0.44721
\cdot \sqrt{b^2-ac}\bigm/\eps,
$$
and this bound is best possible.
\end{theorem}
This theorem is due to
\citet*{A07-pott-00}, except that the
explicit error bound is not stated there.
For comparison, we state the well-known result for convex functions
(see~\citealt{A07-pott-00}, for example):
\begin{theorem}
\label{convex}
If $f$ is
strictly convex or strictly concave $(ac-b^2>0)$, then there is a
piecewise linear function $\hat f$ approximating $f$ with vertical
error $\eps$ that has triangle density
$$
 \frac{2}{\sqrt{27}}\cdot\sqrt{ac-b^2}\cdot \frac1\eps
\approx 0.38490
\cdot\sqrt{ac-b^2}\bigm/\eps.
$$
If $\hat f$ is required to be interpolating, the bound becomes
$$
 \frac{4}{\sqrt{27}}\cdot\sqrt{ac-b^2}\cdot \frac1\eps
\approx 0.76980
\cdot\sqrt{ac-b^2}\bigm/\eps.
$$
These bounds are best possible.
\end{theorem}
\begin{figure}[tb]
  \centering
  \includegraphics[scale=1.2]{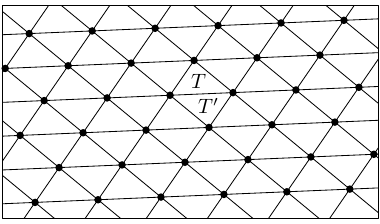}
  \caption{A triangular grid}
  \label{grid}
\end{figure}

As in Theorem~\ref{main}, the triangulations in
Theorems~\ref{interpolating} and~\ref{convex} are triangular grids,
and the statement about the translations and the one-parameter
family of solutions holds likewise.

In contrast to
Theorems~\ref{interpolating} and~\ref{convex}, we do not know whether
the constant in
 Theorem~\ref{main} is best possible.

For an infinite grid like in Figure~\ref{grid}, the number of vertices
per area is half the number of triangles. Thus, in order to get estimates
for the vertex density, just divide the bounds in the theorems by~2.

Our main result,
 Theorem~\ref{main}, will be proved in
Section~\ref{sec:non-interpolating}, after reproving
 Theorem~\ref{interpolating} in
Section~\ref{sec:interpolating}.
For comparison, Section~\ref{sec:convex}
treats the convex case (Theorem~\ref{convex}).
The remainder of this introduction will motivate the problem and
prepare it for the solution.

\subsection{Vertical distance and quadratic functions}
There are two reasons why
 we have chosen to concentrate (i) on quadratic functions and
(ii)~on the vertical distance:
\begin{enumerate}
\item [a)] the relevance from the viewpoint of
applications, and
\item [b)] the mathematical simplicity that comes with this
model and which allows us to derive clean results.
\end{enumerate}
We will discuss these aspects now.

\subsubsection{Applications and related work}
\label{sec:motivation}
Piecewise linear approximation is a fundamental problem in converting
some general function or shape into a form that can be stored and
processed in a computer.  Our original motivation comes from the
desire to approximate the boundaries of three-dimensional configuration
spaces for robot motion planning \citep*{A07-a-phd,A07-agr-pgcss-13},
which turn out to be ruled surfaces with negative Gauss curvature.

Of course, when approximating a surface in space, one does not want to
use the vertical distance but rather something like the Hausdorff
distance, which
measures the distance from the given surface to the \emph{nearest} point
of the approximating surface, in a direction \emph{perpendicular} to
one of the 
surfaces. However, if we consider a small patch of the surface and we
look for a good approximation in a local neighborhood, we can rotate the
surface in 3-space such that it becomes horizontal. Then, as long as the
surface does not curve to much away from the horizontal direction, the
vertical distance is a good substitute for the Hausdorff distance, and
it is always an upper bound on it. 

For 
 piecewise linear approximation,
the first interesting terms of the Taylor approximation are the
quadratic terms.
Thus, quadratic
functions are the model of choice for investigating the question of best
approximation.

Every smooth function can be approximated by a quadratic function in
some neighborhood, and the same is true for surfaces. 
In this sense, our results are applicable as a \emph{local} model, for
a smooth surface or a smooth function as the approximation gets more and more refined.
This approach has been pioneered in the above-mentioned paper of
\citet{A07-pott-00}. Our contribution is to improve the
result for non-interpolating approximation of saddle surfaces.

\citet*{stitching}
have extended this approach to 
an arbitrary bivariate function~$f$, by taking optimal local
approximations on suitably defined patches and ``stitching'' them
together at the patch boundaries. (The setting of this paper is actually
somewhat different: the bivariate function $f$ is given as a set of
scattered data points.)

In arbitrary dimensions, the problem of optimal piecewise linear
approximation has been
adressed 
by
\citet{Clarkson}, without deriving explicit constant factors.
For convex functions and convex bodies, there is a vast literature on
optimal piecewise linear approximation in many variations, see for
example the treatment in~\citet{A07-pott-00} and the references given there.

\subsubsection{Mathematical properties; transforming the problem into
  normal form}

One crucial property of a quadratic function is that, from the point
of view of our problem, it ``looks the same'' everywhere. This is made
precise in the following observation.
\begin{lemma}
\label{same}
Let $f$ be a quadratic function
\begin{equation}
  \tag{\ref{eq:f}}
f(x,y)=ax^2+2bxy+cy^2+dx+ey+g  ,
\end{equation}
 and let
$(x_1,y_1),(x_2,y_2)\in \mathbb{R}^2$ be two points. 
Then there is an affine transformation 
 of $\mathbb{R}^3$ that
\begin{enumerate}
\item maps the graph of $f$ to itself,
\item maps the point
$(x_1,y_1,f(x_1,y_1))$ to the point
$(x_2,y_2,f(x_2,y_2))$,
\item maps vertical lines to vertical lines,
\item leaves vertical distances between points on the same vertical
  line unchanged,
\item acts as a translation in the plane when the $z$-coordinate is ignored.
\end{enumerate}
\end{lemma}
\begin{proof}
We construct a transformation of the form
  \begin{equation*}
    \begin{pmatrix}
      x\\y\\z
    \end{pmatrix}
\mapsto
    \begin{pmatrix}
      x'\\y'\\z'
    \end{pmatrix}
=    \begin{pmatrix}
      1&0&0\\
      0&1&0\\
      u&v&1
    \end{pmatrix}
    \begin{pmatrix}
      x\\y\\z
    \end{pmatrix}
+
    \begin{pmatrix}
      x_2-x_1\\y_2-y_1\\w
    \end{pmatrix}
  \end{equation*}
for some parameters $u,v,w$ that are to be determined.
It is evident that
the points
 $(x_0,y_0,z)$
on a vertical line, for fixed $x_0,y_0$,
 are mapped to points
 $(\bar x,\bar y,\bar z+ z
)$, for some fixed $\bar x,\bar y,\bar z$, and thus,
 Properties 3 and 4 are fulfilled.
Moreover, when restricted to the first two coordinates, the
transformation acts as a translation on the $xy$-plane (Property 5), moving
$(x_1,y_1)$ to $(x_2,y_2)$. Thus, Property 2 holds provided that we
can show Property~1.
Property~1
requires that $f(x,y)=z$ implies $f(x',y')=z'$.
This is fulfilled by setting
$u=
2a
(x_2-x_1)
+2b
(y_2-y_1)
$,
$v=
2b
(x_2-x_1)
+2c
(y_2-y_1)
$,
and $w=
f(x_2-x_1,y_2-y_1)-g
$, as is easily checked by calculation.
\qed
\end{proof}

The problem of finding a best approximation remains also unchanged
when adding a linear function to $f$. This
means that we can assume that $d=e=g=0$ in~\eqref{eq:f}.
By a principal axis transform, the $xy$-plane can be rotated
such that $f$ gets the form
$f(x,y)=a'x^2+ c'y^2$, with $a'c'=ac-b^2$. Finally, we scale the $x$-and $y$-axis by
$\sqrt {|a'|}$ and $\sqrt{| c'|}$ such that $f$ becomes
\begin{equation*}
  f(x,y)=\pm x^2 \pm y^2  .
\end{equation*}
The area is changed by the factor
$\sqrt{|a'c'|}=\sqrt{|ac-b^2|}$,
 and this is taken into account by the corresponding
factor in Theorems~\ref{main}--\ref{convex}.

\subsection{Covering the plane with copies of a triangle}

Now we show that a triangle where all three vertices
have the same fixed offset~$\Delta$ from the surface~$f$ can be used to construct a global
approximation.
The offset
$\Delta$ can be positive or negative.
The case $\Delta=0$ corresponds to interpolating approximation.

\begin{lemma}\label{area}
Let $T=p_1p_2p_3$ be a triangle in the plane with area $A$, and let
$\hat f$ be a linear function such that
$\hat f(p_i)=f(p_i)+\Delta$ for the three vertices $p_i$ of $T$.
Let $\eps$ denote the maximal vertical distance within the triangle\textup:
$\eps := \max\{\,|\hat f(x,y)-f(x,y)| : (x,y)\in T\,\}$.

Then there is a piecewise linear approximation of $f$ over the whole plane with vertical
distance $\eps$ and triangle density $1/A$.
\end{lemma}
\begin{proof}
  If we rotate the triangle $T$ by $180^\circ$ about the origin,
we can define a linear function 
over
 the rotated triangle $T'$ with the same maximum error~$\eps$.
  Translates of $T$ and $T'$ can be used to tile the plane as in
Figure~\ref{grid}.
By Lemma~\ref{same}, defining a linear function over any translate of
$T$ or $T'$ with the same vertex offset $\Delta$ leads to an error
of $\eps$ over this triangle. Since all offsets are equal, the
triangles fit together to form a piecewise linear interpolation over
the whole plane. The triangle density is $1/A$.
\qed
\end{proof}

If we impose the condition that all vertices have the same vertical
offset from the surface $f$, 
this lemma turns the problem of finding
an 
 approximating triangulation with few triangles into the problem of finding a
largest-area triangle $T$ subject to the error bound.

\section{Convex surfaces}
\label{sec:convex}
After these preparations, it is easy to solve the convex case 
$f(x,y)=x^2+y^2$.
Consider first the case of interpolating approximation.
 The largest error over a triangle~$T$ is assumed at the
center of the smallest enclosing circle~$C$, see Figure~\ref{circle}:
\begin{figure}[b]
  \centering
  \includegraphics{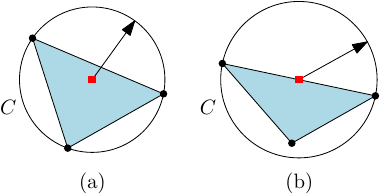}
  \caption{The smallest enclosing circle $C$ for (a) an acute triangle, (b) an obtuse triangle}
  \label{circle}
\end{figure}
This can be seen easily 
when the center of $C$ is at the origin, which we can assume
by Lemma~\ref{same}.
Let $r$ be the radius of $C$.
 Then the line or plane through the points of
maximum distance from the origin lies at height $z=r^2$, giving rise to an error of $\eps=r^2$. The
search for an optimal triangle thus amounts to finding a largest-area
triangle inside a circle of radius~$\sqrt{\eps}$. This is clearly an
equilateral inscribed triangle, and its area is $\eps\frac{3\sqrt3}4$.
By Lemma~\ref{area}, this leads to the second part of Theorem~\ref{convex}.

It is obvious that the optimal triangle is not unique: it can be
rotated freely, giving rise to a one-parameter family of optimal
triangulations.

 The non-interpolating case is now easily derived from
the interpolating case, because for a convex function,
an interpolating approximation
 $\hat f$ cannot
lie below $f$. Thus, finding an interpolating approximation amounts to
looking
for a function $\hat f$ that satisfies
\begin{equation}
\label{inter}
  f(x,y) \le \hat f(x,y) \le f(x,y)+\eps
\end{equation}
for all $x,y$,
see Figure~\ref{conv-approx}.
 To see that the problems are indeed equivalent,
observe that any function $\hat f$ fulfilling \eqref{inter} can be
turned into an interpolating approximation by reducing the values
$\hat f(x,y)$ at each vertex $(x,y)$ of the triangulation to its lower
bound, namely $f(x,y)$, without violating~\eqref{inter}.
On the other hand, non-interpolating approximation looks for a
function
that satisfies
\begin{equation*}
  f(x,y) -\eps \le \hat f(x,y) \le f(x,y)+\eps.
\end{equation*}
A  non-interpolating approximation with error $\eps$ can thus be
obtained from an interpolating approximation $\hat f$ with error $2\eps$ by
subtracting $\eps$ from $\hat f$, and vice versa.

\begin{figure}[htb]
  \centering
  \includegraphics{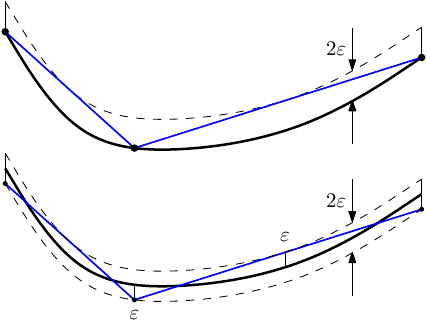}
  \caption{A non-interpolating approximation corresponds to
an interpolating approximation with a doubled
    error bound.}
  \label{conv-approx}
\end{figure}

\section{Saddle surfaces}
\label{sec:negative}

For the indefinite case, it is more convenient to rotate the
coordinate system by $45^\circ$ and consider the function in the form
$$f(x,y)=2xy.$$
\begin{lemma}
  The maximum vertical error between $f(x,y)$ and a linear function
$\hat f(x,y)=ux+vy+w$
over a triangular region $T$ is never attained in the interior of $T$.
\end{lemma}
\begin{proof}
A local maximum of
the error 
$|f(x,y)-f(x,y)|=
\max \{ f(x,y)-\hat f(x,y),\,\hat f(x,y)-f(x,y)\}$
 must 
be a local
maximum of 
the 
 function
$f(x,y)-\hat f(x,y)$ or of $\hat f(x,y)-f(x,y)$.
However, these functions
are saddle functions and they
cannot have a local
  extremum in the interior of $T$.
More formally, they
have,
respectively, the Hessians
$\left(
\begin{smallmatrix}
  0&2\\2&0
\end{smallmatrix}\right)
$
and
$\left(
\begin{smallmatrix}
  0&-2\\-2&0
\end{smallmatrix}\right)
$, which are not negative semi\-definite, and thus the second-order necessary condition for
a local maximum is not fulfilled.
\qed
\end{proof}
We conclude that 
it suffices to measure the approximation error on the edges and vertices
of~$T$.

\subsection{Interpolating approximation}
\label{sec:interpolating}

We will first treat the interpolating case and recover the results of
\citet{A07-pott-00} (our Theorem~\ref{interpolating}) as a
preparation for the free approximation (Theorem~\ref{main}) in
Section~\ref{sec:non-interpolating}. The error along a chord connecting
two points of the surface can be evaluated very easily.
\begin{lemma}
\textup%
{\citep[Lemma~2]{A07-pott-00}}
\label{chord}
  Let $p,q\in \mathbf{R}^2$ be two points. The maximum
vertical error between~$f$ and the linear interpolation between $f(p)$
and $f(q)$ is attained at the midpoint $(p+q)/2$ and its value is
\begin{equation}\label{edge}
\max_{0\le\lambda\le 1}
|
(1-\lambda) f(p) +
\lambda f(q) 
-
f((1-\lambda)p +\lambda q) | = \frac{|f(q-p)|}4.
\end{equation}
\end{lemma}
\begin{proof}
  By Lemma~\ref{same}, we may translate the plane such that $p$ becomes the origin.
Then the function $f$ along the segment $pq$ is simply the quadratic function
$f((1-\lambda)p +\lambda q) = \lambda^2 f(q)$, for which the
statement is easy to establish, see Figure~\ref{fig:quadratic}.
\qed
\end{proof}
\begin{figure}[htb]
  \centering
\includegraphics{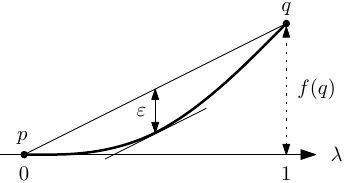}
  \caption{Approximation of a quadratic function by a chord}
  \label{fig:quadratic}
\end{figure}

In the convex case $f(x,y)=x^2+y^2$ of Section~\ref{sec:convex}, it was clear that there is a
one-parametric solution space, since the function is rotationally symmetric.
In the saddle case, it comes somewhat as a surprise that the optimal
triangulations have a similar variability.
 However, this is explained by
the \emph{pseudo-Euclidean transformations}, which have the form
\begin{equation*}
  (x,y) \mapsto (mx,y/m),
\end{equation*}
for a parameter $m\ne 0$, and which leave the graph of the function $f$
invariant, see~\citet{A07-pott-00}.
They scale the $x$- and $y$-coordinates, but they preserve the area.
We will make use of the freedom to apply
{pseudo-Euclidean} transformations
to simplify the calculations.
  
Let us now look for the largest-area triangle $p_1p_2p_3$ such that
the maximum error on each edge $p_1p_2$, $p_2p_3$, and $p_1p_3$, as
computed by \eqref{edge}, is bounded by $\eps$. In more explicit terms, this
means
that 
\begin{equation}
  \label{eq:hyper}
 |f(p_i-p_j)|=|2(x_i-x_j)(y_i-y_j)|\le 4\eps
. 
\end{equation}
We shall show that these constraints must hold as
equalities, because of the concave nature of the constraints.
\begin{lemma}\label{tight}
In a triangle of maximum area subject to the
constraints~\eqref{eq:hyper},
each triangle edge must fulfill this constraint as an
equality.
\end{lemma}
\begin{proof}
Let us consider $p_1$ and $p_2$ as fixed and maximize the area by
varying $p_3$ subject to the constraints~\eqref{eq:hyper}.
We will show that both constraints that involve $p_3$ must be tight.
The lemma then follows by applying the same argument to $p_1$ instead of $p_3$.

The area is a linear function of $p_3$, proportional to the
distance of $p_3$ from the line $p_1p_2$. If none of the two
constraints involving $p_3$ is tight, we can freely move $p_3$ in some
small neighborhood, and hence this situation cannot be optimal.

Suppose now that the constraint for one edge incident to $p_3$, say, $p_1p_3$,
 is tight but the
other one, $p_2p_3$, is fulfilled as a strict inequality.
 The constraint~\eqref{eq:hyper} for $p_1p_3$ confines $p_3$ within a region $R$
bounded by four hyperbola branches centered at $p_1$, shown in Figure~\ref{fig:boundary}a.
\begin{figure}[htb]
  \centering
\noindent
(a)
  \begin{tikzpicture}[scale=0.9]
  \tikzset{plotLabel/.style={%
      postaction={ decorate,transform shape,
        decoration={markings, mark=at position .5 with {\node #1;}}
      }
    }
  }
  \coordinate (p1a) at (0.8-0.5,1/0.8+0.5/0.64);
  \coordinate (p1b) at (0.8+0.45,1/0.8-0.45/0.64);
  \draw [dotted,thick] (p1a) -- (p1b);
\node at (p1b) [below] {$s$};

  \def\bndmax{2.55}
  \def\bndmin{0.39}
  \def\samples{101}
  \def\ptRadius{2}

  \draw[thick,gray,->] (-2.55,0) -- (2.55,0);
  \draw[thick,gray,->] (0,-2.55) -- (0,2.55);
  \begin{scope}[domain=-\bndmax:-\bndmin]
    \tikzset{func/.style={%
        thick,samples=\samples}%
    }%
    \tikzset{col/.style={color=blue}}%
    \tikzset{cCol/.style={color=blue}}
    \draw[func,col,%
    plotLabel={[below=3pt]{$|x y| = 2 \eps$}}%
    ] plot (\x,{1/\x});
    \draw[func,col] plot (-\x,{-1/\x});
    \draw[func,cCol] plot (\x,{-1/\x});%
    \draw[func,cCol] plot (-\x,{1/\x});%
  \end{scope}

  \coordinate (p0) at (0,0);
  \coordinate (p1) at (0.8,1/0.8);

  \fill (p0) circle (\ptRadius pt) node[below left]{$p_1$};
  \fill (p1) circle (\ptRadius pt) node[left]{$p_3\,$};
\node at (-0.6,1) {$R$};
  \begin{scope}[thick]
    \draw [->,shorten >=\ptRadius pt] (p0) -- (p1);
  \end{scope}
\end{tikzpicture}
\qquad
(b)
  \begin{tikzpicture}[scale=0.9]
  \tikzset{plotLabel/.style={%
      postaction={ decorate,transform shape,
        decoration={markings, mark=at position .5 with {\node #1;}}
      }
    }
  }

  \def\samples{101}
  \def\ptRadius{2}

  \draw[thick,gray,->] (-2.55,0) -- (2.55,0);
  \draw[thick,gray,->] (0,-2.55) -- (0,2.55);

  \def\bndmax{2.55}
  \def\bndmin{0.39/1.2}
  \begin{scope}[domain=-\bndmax:-\bndmin]
    \tikzset{func/.style={%
        thick,samples=\samples}%
    }%
    \tikzset{col/.style={color=blue}}%
    \draw[func,col,%
    plotLabel={[below=8pt]{$x y = 2 \eps{-}2\Delta\!$}}%
    ] plot (\x,{1/1.2/\x});
    \draw[func,col
    ] plot (-\x,{-1/1.2/\x});
  \end{scope}
  \def\bndmax{2.55}
  \def\bndmin{0.39*1.2}
  \begin{scope}[domain=-\bndmax:-\bndmin]
    \tikzset{func/.style={%
        thick,samples=\samples}%
    }%
    \tikzset{cCol/.style={color=magenta}}
    \draw[func,cCol,
    plotLabel={[above=8pt]{$x y = -2 \eps{-}2\Delta$}}%
    ] plot (\x,{-1.2/\x});%
    \draw[func,cCol] plot (-\x,{1.2/\x});%
  \end{scope}

  \fill (p0) circle (\ptRadius pt) node[below left]{$p_1$};
\node at (-0.55,0.8) {$R$};
\end{tikzpicture}

  \caption{The feasible region $R$ is bounded by four hyperbola branches.}
  \label{fig:boundary}
\end{figure}
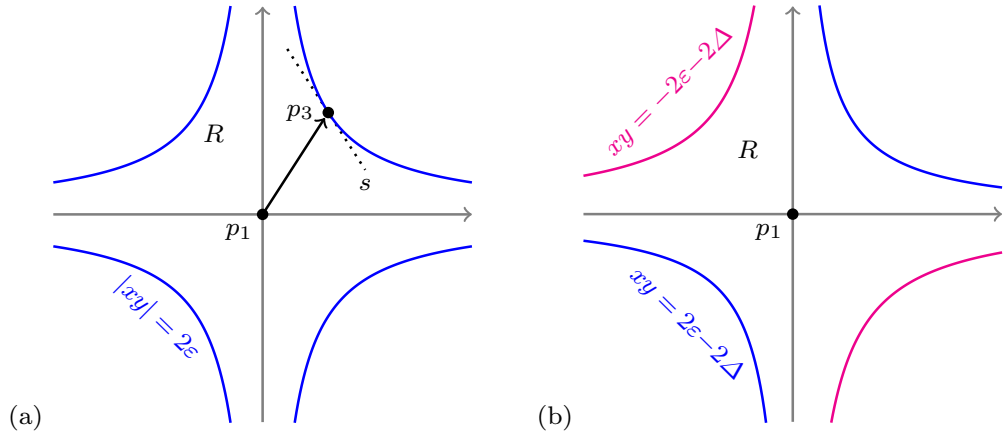
This region is strictly concave in the following sense:
Through each point $p_3\in R$, there is a line segment $s$ in $R$ that
contains $p_3$ in its interior.
(When $p_3$ lies on the boundary of $R$, as we are assuming, $s$ is a
part of
the tangent to the boundary at this point.)
 Moreover, all points of $s$ other than
$p_3$ lie in the interior of $R$. Now, we can move along $s$ by some
small amount in at
least one
direction without decreasing the area function, such that, in the resulting
point, none of the two constraints involving $p_3$ is tight. But this
was already excluded above.
\qed
\end{proof}

We can classify
 triangle edges $p_ip_j$ 
into \emph{ascending edges}
 (extending in the SW--NE direction)
and \emph{descending edges}
 (extending in the NW--SE direction), according to the sign of
$(x_j-x_i)(y_j-y_i)$. 
 Let us assume without loss of generality that the
predominant category is ascending, and two ascending
 edges are $p_1p_2$ and
$p_1p_3$.
Furthermore,
by Lemma~\ref{same}, we can assume that
$p_1$ is at the origin, and, after a rotation by $180^\circ$ if
necessary, $p_2$ and $p_3$ lie in the first quadrant, see
Figure~\ref{fig:hyperbolas}a.
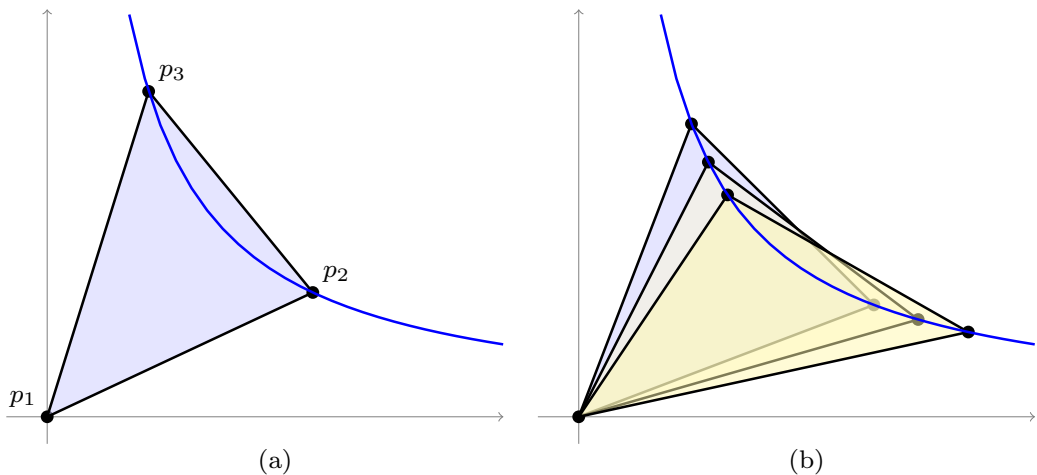
\begin{figure}[htb]
  \centering

  \def\ptRad{2pt} 
  \def\golden{1.61803} 
  \def\eps{1.0}
  \pgfmathsetmacro{\xii}{2*\golden*sqrt(\eps)}

\newif\iflabels
\labelstrue

  \newcommand*{\drawperturbedtriangle}[2]{
    \path let \p1 = (p1) in coordinate (P1) at (#1*\x1,\y1/#1);
    \path let \p1 = (p2) in coordinate (P2) at (#1*\x1,\y1/#1);
    \draw[thick,fill=#2,fill opacity=0.5]

    (p0) node[opacity=1,black,above left]{\iflabels$p_1$\fi} --
    (P1) node[opacity=1,black,above right]{\iflabels$p_2$\fi} --
    (P2) node[opacity=1,black,above right]{\iflabels$p_3$\fi} -- cycle;

    \foreach \i in {1,2}
    \fill (P\i) circle (\ptRad);
  }%

\noindent
\begin{tikzpicture}
  \draw[help lines,->] (-1.5em,0) -- (5,0);
  \draw[help lines,->] (0,-1em) -- (0, 4.5);

\coordinate (p0) at (0,0);
\coordinate (p1) at ($({\xii},{4*\eps/\xii})$);
\coordinate (p2) at ($({\xii/pow(\golden,2)},{4*\eps*pow(\golden,2)/\xii})$);
 \foreach \i in {0}
 \fill (p\i) circle (\ptRad);

    \drawperturbedtriangle{0.9}{blue!20}

  \draw[thick,blue,domain=0.9:5] plot (\x,{4*\eps/\x});
 \node at (2.5,-0.5) {(a)};
\end{tikzpicture}
\hfil
\labelsfalse
\begin{tikzpicture}
  \draw[help lines,->] (-1.5em,0) -- (5,0);
  \draw[help lines,->] (0,-1em) -- (0, 4.5);

\coordinate (p0) at (0,0);
\coordinate (p1) at ($({\xii},{4*\eps/\xii})$);
\coordinate (p2) at ($({\xii/pow(\golden,2)},{4*\eps*pow(\golden,2)/\xii})$);
 \foreach \i in {0}
 \fill (p\i) circle (\ptRad);

    \drawperturbedtriangle{1}{blue!20}
    \drawperturbedtriangle{1.15}{yellow!20}
    \drawperturbedtriangle{1.32}{yellow!40}
 \node at (2.5,-0.5) {(b)};
  \draw[thick,blue,domain=0.9:5] plot (\x,{4*\eps/\x});
\end{tikzpicture}

  \caption{(a) The potential positions of $p_2$ and $p_3$.
(b)~Transforming the triangle by pseudo-Euclidean motions.}
  \label{fig:hyperbolas}
\end{figure}
 The two points $p_2=(x_2,y_2)$ and $p_3=(x_3,y_3)$ lie on 
the hyperbola $xy=2\eps$. By applying a pseudo-Euclidean
transformation, it suffices to consider the case that they lie symmetric with respect to
the line $x=y$, i.e., 
$(x_3,y_3)=(y_2,x_2)$.
We now substitute this
and the equation $y_2=2\eps/x_2$
 into
the relation~\eqref{eq:hyper} for
$|f(p_3-p_2)|$ and obtain the equation
$|f(p_3-p_2)|
=|2(x_3-x_2)(y_3-y_2)|
=|2(y_2-x_2)(x_2-y_2)|
=2(y_2-x_2)^2
=2(2\eps/x_2-x_2)^2
=4\eps
$. Solving for $x_2$ gives 
$p_2=(
\sqrt{\eps/2}(\sqrt5+1),
\sqrt{\eps/2}(\sqrt5-1))$.
(The quartic equation for $x_2$ has four solutions in total. There is another nonnegative
solution, which just swaps the two coordinates $x_2$ and $y_2 $, or
equivalently, swaps $p_2$ with $p_3$. The other two solutions are just
the negations of the first two.) The area of the triangle
$p_1p_2p_3$ is
\begin{equation}
  \label{eq:triangle-area}
\frac 12
\left|
  \begin{matrix}
    x_2&x_3\\
    y_2&y_3\\
  \end{matrix}
\right| =
\frac 12
\left|
  \begin{matrix}
    x_2&y_2\\
    y_2&x_2\\
  \end{matrix}
\right| =
\tfrac12(x_2^2-y_2^2) =
\tfrac12(x_2+y_2)(x_2-y_2)
 =
 \eps\sqrt5.
\end{equation}
By Lemma~\ref{area}, this establishes Theorem~\ref{interpolating}.
\qed
\medskip

The one-parameter family of triangulations that are optimal is
obtained by applying pseudo-Euclidean transformations to the symmetric
solution $p_1p_2p_3$ computed above, see
Figure~\ref{fig:hyperbolas}b: The set of optimal triangles
consists of all
triangles $p_1p_2p_3$ with $p_1=(0,0)$,
$p_2=(
\sqrt{\eps/2}(\sqrt5+1)\cdot m,
\sqrt{\eps/2}(\sqrt5-1)/m)$,
and
$p_3=(
\sqrt{\eps/2}(\sqrt5-1)\cdot m,
\sqrt{\eps/2}(\sqrt5+1)/m)$,
for $m\ne 0$, as well as their reflections in the coordinate axes and
their translations.

We can use this freedom to choose a triangulation
which optimizes some secondary criterion, like the shape of the triangles.
\citet[Chapter~3]{A07-a-phd} considered the problem of
maximizing the smallest angle. He showed that the optimal triangle is,
not surprisingly, always an isosceles triangle. In general, there are two different
shapes of isosceles triangles, corresponding to the two patterns in
Figure~\ref{fig:lifted}a--b. For a general quadratic function,
these two cases have differently shaped triangles, and the best choice
depends on the ratio between the eigenvalues of the quadratic form
associated to~$f$.

 The family of optimal triangulations that are characterized above
 is somewhat counter-intuitive: The surface
described by $z=2xy$ is a \emph{ruled} surface: it is swept out by a
line.  Any edge between two points on a line of the ruling has error
0, no matter how long it is. It seems attractive to use edges
that go along the ruling.  The above results show that this is not the
best idea: it is better to ``distribute'' the error evenly to the three
edges. If one wants to insist on following the ruling, one can impose that $p_2$ lies on
the $x$-axis. The optimal triangle is then an isosceles
triangle $p_1p_2p_3$ with a base
$p_1p_2$ of arbitrary length
and $p_3$ on the hyperbola $|xy|=2\eps$,
and its area is~$2\eps$ instead of~$\eps\sqrt5$.

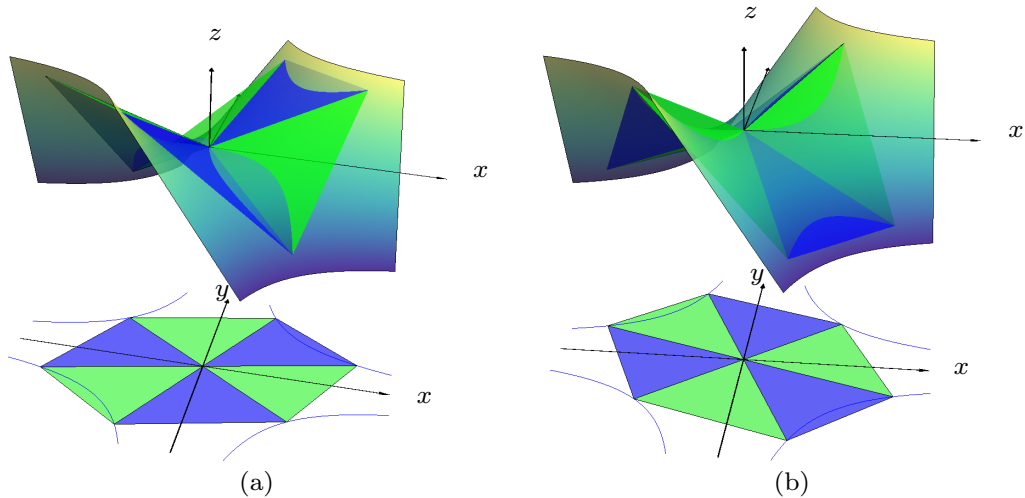
\begin{figure}[htb]
  \centering
  \begin{tabular}{cc}
\rlap{\includegraphics[width=4.8cm]
{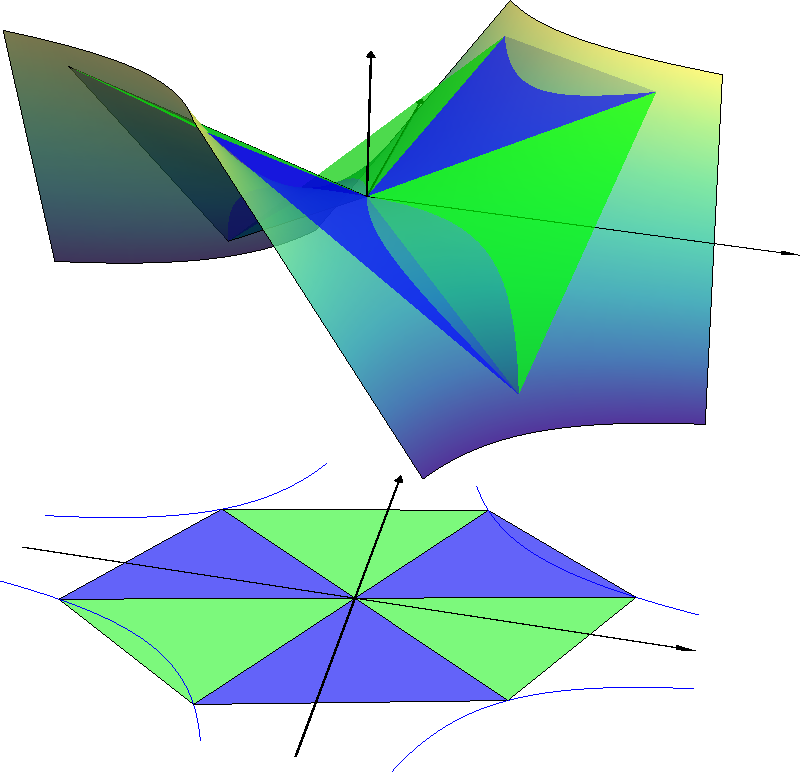}}%
\begin{picture}(155,0)
  \put(127.5,18){$x$}
  \put(145,89){$x$}
  \put(62.5,132){$z$}
  \put(64.5,51.5){$y$}
\end{picture}
&
\rlap{\includegraphics[width=4.8cm]
{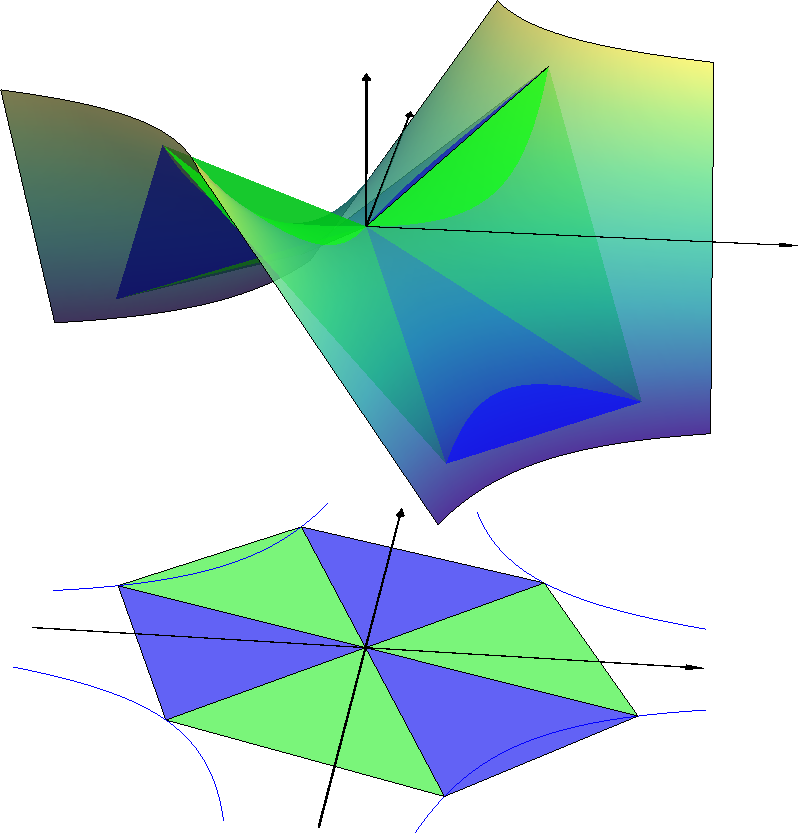}}%
\begin{picture}(155,0)
  \put(128,27){$x$}
  \put(145,101.6){$x$}
  \put(62.5,139){$z$}
  \put(64.5,56.5){$y$}
\end{picture}
\\(a)&(b)\\    
  \end{tabular}

  \caption{The triangles surrounding the origin in an optimal
interpolating
    triangulation of a saddle surface. Depending on the chosen
    orientation, 
 the
approximating triangulation 
will
predominantly lie (a) above the surface or (b) below the surface.}
  \label{fig:lifted}
\end{figure}

Incidentally, the same fallacious line of
reasoning has led L. Fejes T\'oth,
in his celebrated book \emph{Lagerungen in der Ebene, auf der Kugel
  und im Raum} from 1953~\citep[Section~V.12, p.~151]
{FT}, to assert erroneously that 
a
ruled surface such as a hyperboloid of one sheet
could be triangulated with
a triangle density of only $O(1/\sqrt\eps)$.
 For more details,
see
~\citet[Section~2.4]{wintraecken}.

\subsection{Non-interpolating approximation}
\label{sec:non-interpolating}

{In this section, we will prove our main result,
Theorem~\ref{main}, which concerns non-interpolating approximation.
\looseness-1%
When we try to improve the approximation by allowing the vertices of
the triangles to move away from the surface, we encounter a challenge:
In contrast to the convex case, some edges (the ascending ones) run
above $f$ and others (the descending ones) run below~$f$.
Figure~\ref{fig:lifted} shows that the linear approximation penetrates
the surface $f$, lying partially above it and below it.  It is
 not clear in which direction one should start moving the
vertices to improve the approximation. Accordingly,
\citet*{A07-pott-00}
 conjectured that the best approximation is
the  interpolating approximation.
We will see that this is not the case.

}

To make the problem manageable,
we impose the following constraint: Every vertex of the
approximations has the same offset $\Delta$ (positive or negative)
from the surface. This ensures that we can apply Lemma~\ref{area}, 
and it suffices to look for one triangle that maximizes the area.

Lemma~\ref{chord} must be modified to take into account the vertical
shift by $\Delta$.

\begin{lemma}\label{chord2}
  Let $p,q\in \mathbf{R}^2$ be two points. The maximum
vertical error between~$f$ and the linear interpolation between $f(p)
+\Delta
$
and $f(q)+\Delta$ is attained either at the midpoint $(p+q)/2$ or at
the endpoints $p$ and $q$, and its value is
\begin{equation*}
\max_{0\le\lambda\le 1}
|
(1-\lambda) f(p) +
\lambda f(q) + \Delta
-
f((1-\lambda)p +\lambda q) | = 
\max \{|\Delta|, |\Delta +\tfrac{f(q-p)}4|\}.
\eqno\qed
\end{equation*}
\end{lemma}
Let us fix a point $p_1$ and
ask for the possible
locations
of a point $p_2$ such that the approximation error on the edge
$p_1p_2$
does not exceed $\eps$.
Assuming that $|\Delta|\le\eps$,
we can rewrite the condition $|\Delta +{f(q-p)}/4|\le\eps$, and
we see that the vector $(x,y)=(x_2-x_1,y_2-y_1)$ must satisfy the inequalities
$$
-(\eps +\Delta) \le xy/2 \le  \eps -\Delta .
$$
This is a region bounded by two different hyperbolas, see
Figure~\ref{fig:boundary}b,
 but
the arguments from the previous section about the concavity of the
region remain valid, showing that the error must be attained at all
three edges
(Lemma~\ref{tight}). As before, we can also assume that $p_1$ lies at the
origin and $p_2$ and $p_3$ lie in the first quadrant.
(To achieve the last situation, we may have to switch the sign of~$\Delta$.)
In this situation, $f(x_2,y_2)$ and
 $f(x_3,y_3)$ are positive, and we have
$x_2y_2 =
x_3y_3 = 2(\eps-\Delta )
$. Again, by a pseudo-Euclidean transformation, we simplify the
computation by assuming the symmetric situation
$(x_3,y_3)=(y_2,x_2)$.
 The third edge $p_2p_3$ is descending, because it is a chord of the
 hyperbola in the first quadrant.
Thus, with $p_3-p_2=(x_3-x_2,y_3-y_2)$,
the quadratic function
$f(p_3-p_2)$ will be negative, and we get the equation
$f(p_3-p_2)=-4(\eps+\Delta)$.
The term $f(p_3-p_2)$ evaluates to
$f(p_3-p_2)=-2(y_2-x_2)^2
=-2(2(\eps-\Delta )/x_2-x_2)^2$.
Solving the resulting quadratic equation
$$
x_2^2 \pm \sqrt{2(\eps+\Delta)}x_2 - 2(\eps-\Delta)=0
$$
 gives
\begin{equation*}
  p_2=(x_2,y_2)=
\Bigl(\sqrt{\tfrac12}\cdot
(\sqrt{5\eps-3\Delta}\pm \sqrt{\eps+\Delta}),
\,
\sqrt{\tfrac12}\cdot
(\sqrt{5\eps-3\Delta}\mp \sqrt{\eps+\Delta})
\Bigr).
\end{equation*}
As in
\eqref{eq:triangle-area}, the area of a symmetric triangle $(0,0)$,
($x_2,y_2)$,
 ($y_2,x_2)$ is
$\tfrac12|(x_2+y_2)(x_2-y_2)|$. This evaluates to
$\sqrt{5\eps-3\Delta}\cdot\sqrt{\eps+\Delta}
=\sqrt3\cdot \sqrt{5\eps/3-\Delta}\cdot\sqrt{\eps+\Delta}$,
and this is maximized for $\Delta = \eps/3$, yielding an area
of $4/\sqrt3\cdot\eps$.
The necessary condition $|\Delta|\le\eps$ is fulfilled.
By Lemma~\ref{area}, this establishes Theorem~\ref{main}.
\qed

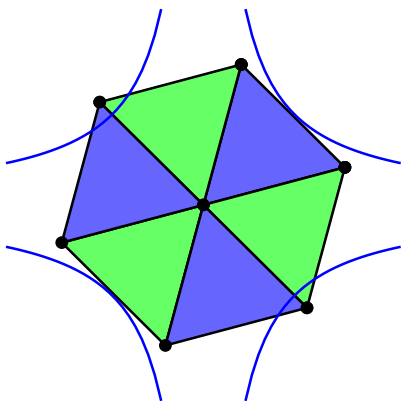
\begin{figure}[htb]
  \centering
  \begin{tikzpicture}[thick]
  \def\eps{0.25}
  \def\ptRad{2pt}

  \coordinate (p1) at (0,0);
  \coordinate (p2) at ($({3.10421*sqrt(\eps)},{0.83177*sqrt(\eps)})$);
  \coordinate (p3) at ($({0.83177*sqrt(\eps)},{3.10421*sqrt(\eps)})$);
  \coordinate (p4) at ($({-2.27244*sqrt(\eps)},{2.27244*sqrt(\eps)})$);
  \coordinate (p5) at ($({2.27244*sqrt(\eps)},{-2.27244*sqrt(\eps)})$);
  \coordinate (p6) at ($({-3.10421*sqrt(\eps)},{-0.83177*sqrt(\eps)})$);
  \coordinate (p7) at ($({-0.83177*sqrt(\eps)},{-3.10421*sqrt(\eps)})$);

  \draw[fill=blue!60] (p1) -- (p2) -- (p3) -- cycle;
  \foreach \i in {1,2,3}
  \fill (p\i) circle (\ptRad);

  \draw[fill=green!60] (p1) -- (p3) -- (p4) -- cycle;
  \draw[fill=green!60] (p1) -- (p2) -- (p5) -- cycle;
  \foreach \i in {1,2,3,4,5}
  \fill (p\i) circle (\ptRad);

  \draw[fill=green!60] (p1) -- (p6) -- (p7) -- cycle;
  \draw[fill=blue!60] (p1) -- (p5) -- (p7) -- cycle;
  \draw[fill=blue!60] (p1) -- (p6) -- (p4) -- cycle;
  \foreach \i in {1,2,...,7}
  \fill (p\i) circle (\ptRad);

  \pgfmathsetmacro{\minimalxvalue}{0.83177*sqrt(\eps)/0.9}
  \pgfmathsetmacro{\maximalxvalue}{4*sqrt(\eps)/0.83177*0.9}
  \draw[thick,blue,domain=\minimalxvalue:\maximalxvalue] plot (\x,{4*\eps/\x});
  \draw[thick,blue,domain=\minimalxvalue:\maximalxvalue] plot (\x,{-4*\eps/\x});
  \draw[thick,blue,domain=-\minimalxvalue:-\maximalxvalue] plot (\x,{4*\eps/\x});
  \draw[thick,blue,domain=-\minimalxvalue:-\maximalxvalue] plot (\x,{-4*\eps/\x});
\end{tikzpicture}
  \caption{Optimal non-interpolating triangles, together with the
    hyperbolas
$|xy|=2\eps$.}
  \label{fig:optimal}
\end{figure}

Figure~\ref{fig:optimal} shows six reflected copies
of the optimal triangle
 surrounding the origin, together with
the hyperbolas
$xy=\pm2\eps$
 that were used to define the optimal interpolating
triangulation.
The optimal triangle turns out to be an \emph{equilateral} triangle
(of side length $\sqrt{8\eps/3}$), and it happens to touch the hyperbola
$xy=2\eps$.
We do not have an explanation for these 
 phenomena.

\citet{A07-pott-00}
proposed to call the optimal triangles for the interpolating
approximation of the function $xy$, as defined in
Section~\ref{sec:interpolating}, the \emph{equilateral} triangles of
pseudo-Euclidean geometry.
Maybe it would be more appropriate to reserve this name
for the triangles of Figure~\ref{fig:optimal} and their
pseudo-Euclidean transformations, in view of their remarkable properties.

\section{Concluding remarks}
\label{sec:conclusion}

The constants in
 Theorems~\ref{main} and~\ref{interpolating} are very close. Thus,
the freedom to use non-interpolating approximations seems to give
only a slight improvement.
This is in
 contrast to the case of convex functions
(Theorem~\ref{convex}), where
non-interpolating approximations are better 
by a factor of~2.

The optimality of the approximations found in Theorem~\ref{main} remains
open. If different vertices have different offsets, one is forced to
use more than just one type of triangle, and the situation becomes complicated.

In contrast to this, we know that the constants in
Theorems~\ref{interpolating} and~\ref{convex}
about non-convex interpolating approximation and about convex interpolations are best possible.
The expressions of
Theorems~\ref{interpolating} and~\ref{convex}
hold therefore as
lower bounds on the triangle density of any triangulation.
This follows from the proofs,
because a triangulation of density $D$ must contain triangles of
area at least
$1/D-\delta$ for arbitrarily small $\delta>0$.
 These lower bounds also
apply when we want to triangulate a bounded polygonal domain~$\Omega$. The
triangle density is then simply the number of triangles divided by the
area of $\Omega$.
 In this case, the bound cannot be achieved in general, because the grid has to adapt to
the boundary of $\Omega$, but it can be attained
asymptotically as $\eps\to0$.

Another question that would be worth while to be attacked would be good (or
optimal) triangulations for trivariate quadratic functions.


\bibliographystyle{plainnat}
\bibliography{saddletriang}

\end{document}